\documentclass[12pt]{article}

\usepackage{amsmath,amssymb,amsthm}

\def\stf#1#2{\left[#1\atop#2\right]}

\newtheorem{theorem}{Theorem}
\newtheorem{Prop}{Proposition}
\newtheorem{Cor}{Corollary}
\newtheorem{Lem}{Lemma}

\allowdisplaybreaks

\begin{document}

\title{Polynomial identities and Fermat quotients
}

\author{
Takao Komatsu
\\
\small Department of Mathematical Sciences, School of Science\\[-0.8ex]
\small Zhejiang Sci-Tech University\\[-0.8ex]
\small Hangzhou 310018 China\\[-0.8ex]
\small \texttt{komatsu@zstu.edu.cn}\\\\
B. Sury\\
\small Stat-Math Unit\\[-0.8ex]
\small Indian Statistical Institute\\[-0.8ex]
\small 8th Mile Mysore Road\\[-0.8ex]
\small Bangalore 560059 India\\[-0.8ex]
\small \texttt{sury@isibang.ac.in}
}

\date{
\small MR Subject Classifications: Primary 11B65; Secondary 11A07, 05A10, 11B50, 11B73
}

\maketitle

\begin{abstract}
We prove some polynomial identities from which we deduce congruences
modulo $p^2$ for the Fermat quotient $\frac{2^p-2}{p}$ for any odd
prime $p$ (Proposition 1 and Theorem 1). These congruences are
simpler than the one obtained by Jothilingam in 1985
(\cite{Jothilingam}) which involves listing quadratic residues in
some order. On the way, we also observe some more congruences for
the Fermat quotient that generalize Eisenstein's classical
congruence (Lemma 1). Using such polynomial identities, we obtain
some sums involving harmonic numbers. We also prove formulae for
binomial sums of harmonic numbers of higher order (Theorem 2).
\end{abstract}

Throughout, we use the standard notation $\frac{a}{b} \equiv
\frac{c}{d}$ modulo $m$ for a positive integer $m$ relatively prime
to $bd$ if, $\frac{a}{b}-\frac{c}{d}=\frac{mp}{q}$ with
$\gcd(m,q)=1$.

\section{Generalizing Eisenstein's congruence for Fermat quotients}

\noindent Eisenstein had proved for an odd prime that, modulo $p$,
we have
$$\frac{2^{p-1}-1}{p} \equiv \sum_{j:odd,~j<p-1} \frac{1}{j}.$$
It is possible to obtain this and some more congruences by
exploiting an elementary polynomial identity as follows.

\noindent {\bf Observation.}\\
For any integer $m \geq 1$, we have the evident polynomial identity
\begin{equation}
\frac{x^m-x}{m} =\sum_{r=1}^{m-1} {m \choose r}
\frac{(x+1)^r(-1)^{m-r}}{m} +\frac{(x+1)^m-x+(-1)^m}{m}.
\label{eq:0}
\end{equation}

As a consequence, we observe:

\begin{Lem}
For any odd prime $p$, we have the following congruence modulo $p$:
\begin{align}
\frac{2^{p-1}-1}{p} &\equiv \frac{1}{2}
\sum_{j=0}^{p-2}\frac{1}{{p-2 \choose j}}. \label{eq:E}
\end{align}
Further, for any integer $n \geq 2$ and any odd prime $p$, we have
the following congruences modulo $p$:
\begin{equation}
\frac{n^p-n}{p} \equiv - \sum_{r=1}^{p-1} \frac{1^r+2^r + \cdots +n^r}{r}.
\label{eq:A}
\end{equation}
\begin{equation}
\frac{n^p-n}{p} \equiv - \sum_{r=1}^{p-1} \frac{(-1)^r(1^r+2^r + \cdots+ (n-1)^r)}{r}.
\label{eq:AA}
\end{equation}
In particular, modulo $p$, we have:
\begin{align}
\frac{2^{p-1}-1}{p} &\equiv - \frac{1}{2} \sum_{j=1}^{p-1}\frac{2^j}{j},
\label{eq:B}\\
\frac{2^{p-1}-1}{p} &\equiv - \frac{1}{2}
\sum_{j=1}^{p-1}\frac{(-1)^j}{j}. \label{eq:C}\\
\frac{2^{p-1}-1}{p} &\equiv \sum_{j{\rm :odd},~j<p-1}\frac{1}{j}
\equiv - \frac{1}{2} \sum_{j=1}^{(p-1)/2} \frac{1}{j} \label{eq:D}
\end{align}
\end{Lem}

\noindent {\bf Proof.} It follows from (\ref{eq:0}) for $m=p$ that:
$$
\frac{a^p-a}{p} = \frac{(a+1-1)^p-a}{p} =\sum_{r=1}^{p-1} {p \choose r} \frac{(a+1)^r(-1)^{p-r}}{p} +\frac{(a+1)^p-(a+1)}{p}.
$$
We claim that, modulo $p$,
\begin{equation}
\frac{a^p-a}{p} \equiv \sum_{r=1}^{p-1}\frac{(a+1)^r}{r} + \frac{(a+1)^p-(a+1)}{p}.
\label{spadesuit}
\end{equation}
In the above, we used the observation that if $p$ is an odd prime and $0< r <p$, then the integer $\frac{1}{p} {p \choose r} \equiv \frac{(-1)^{r-1}}{r}\mod p$.
This is so because
$$
\frac{1}{p} {p \choose r} = \frac{(p-1)(p-2) \cdots (p-r+1)}{r!}
\equiv \frac{(-1)^{r-1}(r-1)!}{r!} = \frac{(-1)^{r-1}}{r}.
$$
Putting
$a=0$ gives the well-known congruence $\sum_{r=1}^{p-1} \frac{1}{r} \equiv 0\mod p.$\\
Thus, congruence (6) is an equivalent version of Eisenstein's congruence.\\
Putting $a=1$ gives congruence (\ref{eq:B}).\\
Putting $a=-2$ gives the congruence (\ref{eq:C}).\\
Inductively, from (\ref{spadesuit}), one gets then that
$$
\frac{n^p-n}{p} \equiv - \sum_{r=1}^{p-1} \frac{2^r + \cdots +n^r}{r}\mod p.
$$
When $p$ is an odd prime, $\sum_{r=1}^{p-1} \frac{1}{r}\equiv 0\mod p$ (indeed, it is even zero modulo $p^2$ when $p>3$ by Wolstenholme's theorem).
Thus, we have the more symmetric form asserted as (\ref{eq:A}).
Finally, (\ref{eq:AA}) is gotten similarly to (\ref{eq:A}) inductively from (\ref{spadesuit}) by putting $a = -2, -3, -4$ etc.\\
Clearly, using the fact that $\sum_{j \leq p-1} \frac{1}{j} \equiv 0\mod p$, (\ref{eq:C}) implies the two congruences in (\ref{eq:D}).
To prove the congruence (\ref{eq:E}), let us use the following identities which were proved in \cite{Sury93}:
$$
\sum_{r=0}^n \frac{1}{{n \choose r}} = \frac{n+1}{2^n} \sum_{i=0}^n
\frac{2^i}{i+1} = \frac{n+1}{2^n} \sum_{j{\rm :odd}} {n+1 \choose j}
\frac{1}{j}.
$$
If we put in $n=p-2$ for a prime $p \geq 3$, the second expression becomes
$$
\frac{p-1}{2^{p-2}} \sum_{i=0}^{p-2} \frac{2^i}{i+1} =
\frac{p-1}{2^{p-1}} \sum_{j=1}^{p-1} \frac{2^j}{j} \equiv -\sum_{j=1}^{p-1} \frac{2^j}{j},
$$
which gives the congruence (\ref{eq:E}) on using (\ref{eq:B}). This
completes the proof of the lemma.

\section{New congruences for Fermat quotient modulo $p^2$}

\noindent We prove some polynomial identities which are then used to
obtain two different congruences for Fermat quotients modulo $p^2$;
these are simpler than the ones obtained in \cite{Jothilingam}.

\begin{Lem} For any odd positive integer $n$,
$$
\sum_{r=1}^n \frac{(-1)^{r-1}x^r}{r} = \sum_{r=1}^n {n \choose r}
\frac{(-1)^{r-1}(x+1)^r}{r} + \sum_{r=1}^n \frac{(-1)^r}{r}
{n\choose r}.
$$
\end{Lem}

\noindent {\bf Proof.} Consider the evident identity
$$
x^n =  \sum_{r=1}^n {n \choose r}(x+1)^r (-1)^{r-1}-1.
$$
As $n$ is odd, we have
$$
1-x+x^2- \cdots + x^{n-1}= \frac{x^n+1}{x+1}= \sum_{r=1}^n {n\choose r} (-1)^{r-1} (x+1)^{r-1}.
$$
Integration gives
$$
x - \frac{x^2}{2} + \cdots + \frac{x^n}{n}
= \sum_{r=1}^{n} {n \choose r} (-1)^{r-1} \frac{(x+1)^r}{r} + c
$$
for some constant $c$. Putting $x=0$, we obtain $c=
\sum_{r=1}^n\frac{(-1)^r}{r} {n \choose r}$. Thus, we have the
asserted polynomial identity:
$$
\sum_{r=1}^n \frac{(-1)^{r-1}x^r}{r} = \sum_{r=1}^n {n \choose r}
\frac{(-1)^{r-1}(x+1)^r}{r} + \sum_{r=1}^n \frac{(-1)^r}{r} {n\choose r}.
$$

\noindent Note in passing that  by comparing coefficients of $x^2$
and $x^3$, we get
\begin{align*}
n &= \sum_{r=1}^{n-1} (-1)^r {n \choose r} (r-1);\\
(n-1)(n-2) &= 2 + \sum_{r=1}^{n-1} (-1)^r {n \choose r}(r-1)(r-2).
\end{align*}

\noindent Now, we use the above polynomial identity to obtain two
different congruences for the Fermat quotient modulo $p^2$ for an
odd prime $p$.

\begin{Prop}
For an odd prime $p$,
$$
\frac{2^{p-1}-1}{p} \equiv  \sum_{r=1}^{p-1} \frac{-2^{r-1}}{r} \mod
{p^2}.
$$
\end{Prop}
\begin{proof}
In the polynomial identity
$$
\sum_{r=1}^n \frac{(-1)^{r-1}x^r}{r} = \sum_{r=1}^n {n \choose r}\frac{(-1)^{r-1}(x+1)^r}{r} + \sum_{r=1}^n \frac{(-1)^r}{r} {n\choose r},
$$
take $x=-2$. We obtain
$$
-\sum_{r=1}^p \frac{2^r}{r} = -\sum_{r=1}^{P-1} {p \choose r}
\frac{1}{r} - \frac{2}{p} + \sum_{r=1}^{p-1} {p \choose r}
\frac{(-1)^r}{r}.
$$
Rewriting this, we have
$$
\frac{2^p-2}{p} + \sum_{r=1}^{p-1} \frac{2^r}{r} = 2 \sum_{r{\rm :odd},~r<p} {p \choose r} \frac{1}{r}.
$$
Firstly, for $p=3$, the proposition follows by direct computation.
Therefore, the proposition will follow if we show that for $p \geq 5$,
$$
\sum_{r{\rm :odd},~ r<p} {p \choose r} \frac{1}{pr} \equiv 0\mod p.
$$
Now for $1 \leq r \leq p-2$ with $r$ odd, we have, modulo $p$,
$$
\frac{1}{pr} {p \choose r} = \frac{(p-1)(p-2) \cdots(p-r+1)}{r^2} \equiv \frac{1}{r^2}
$$
since $r$ is odd. Therefore,
$$
\sum_{r{\rm :odd},~ r<p} {p \choose r} \frac{1}{pr} \equiv
\sum_{r{\rm :odd},~r<p} \frac{1}{r^2} \equiv \frac{1}{2} \sum_{r=1}^{p-1}
\frac{1}{r^2}
$$
since $(p-r)^2 \equiv r^2\mod p$ and $p-r$ runs through the even integers $<p$ when $r$ runs through the odd integers $<p$. But,
$$
\sum_{r=1}^{p-1} \frac{1}{r^2} \equiv \sum_{d=1}^{p-1} d^2 =
\frac{(p-1)p(2p-1)}{6} \equiv 0
$$
if $p \geq 5$.
Therefore, the proposition is proved.
\end{proof}

\noindent We have already proved a congruence for $\frac{2^p-2}{p}$
modulo $p^2$. In 1985, Jothilingam \cite{Jothilingam} had proved a
congruence that involves an ordered choice of quadratic residues.
Below, we prove a different, simpler congruence. \vskip 3mm

\begin{Lem}
$\sum_{r=0}^{n-1} \frac{(1-x)^r}{r+1} = \sum_{r=0}^{n-1} {n \choose
r+1} \frac{(-1)^r}{r+1} \frac{x^{r+1}-1}{x-1}.$
\end{Lem}

\noindent {\bf Proof.} Start with the elementary polynomial identity
$$
- \sum_{k=1}^n (1-x)^{k-1} = \sum_{k=1}^n {n \choose k} (-1)^k
x^{k-1}.
$$
Integrating this, we have
$$
\sum_{k=1}^n \frac{(1-x)^k}{k} = \sum_{k=1}^n {n \choose k}\frac{(-1)^k (x^k-1)}{k}.
$$
The above identity has been written after finding the constant of
integration by putting $x=1$. Rewriting the above identity by taking
$k=r+1$, we have the asserted polynomial identity:
\begin{equation}
\sum_{r=0}^{n-1} \frac{(1-x)^r}{r+1} = \sum_{r=0}^{n-1} {n \choose r+1} \frac{(-1)^r}{r+1} \frac{x^{r+1}-1}{x-1}.
\label{heartsuit}
\end{equation}

\noindent As an application, we get:

\begin{theorem}
For any odd prime $p$, we have
$$
\frac{2^{p-1}-1}{p} \equiv    \sum_{r{\rm :odd},~r<p} \frac{1}{r}  -
p \sum_{r=1}^{p-1} \frac{2^{r-1}}{r^2} \mod{p^2}.
$$
\end{theorem}
\begin{proof}
For any $n\geq 1$, integrating (\ref{heartsuit}) we have
$$
\sum_{r=0}^{n-1} \frac{(1-x)^{r+1}}{(r+1)^2} = \sum_{r=0}^{n-1} {n\choose r+1} \frac{(-1)^{r+1}}{r+1} \left(x+ \frac{x^2}{2}+ \cdots +\frac{x^{r+1}}{r+1}\right) + C,
$$
where the constant $C$ is obtained by putting $x=1$.
We obtain $C = \sum_{r=0}^{n-1} {n \choose r+1}\frac{(-1)^{r}}{r+1} (1+ \frac{1}{2}+ \cdots + \frac{1}{r+1})$.\\
Let us consider $x=-1$ and $n=p$, for an odd prime $p$. Since ${p\choose r+1} \equiv 0\mod p$ for $r<p-1$, we have
$$
\sum_{r=0}^{p-1} \frac{2^{r+1}}{(r+1)^2} \equiv \frac{2}{p}
\sum_{d{\rm :odd},~ d \leq p} \frac{1}{d}\mod p.
$$
Clearly, this is the congruence
$$
\frac{2-2^p}{p^2} \equiv  \sum_{r=1}^{p-1} \frac{2^r}{r^2} -\frac{2}{p} \sum_{r{\rm :odd},~ r<p} \frac{1}{r}\mod p,
$$
which gives, on multiplying by $p$, the asserted congruence modulo $p^2$ in the proposition.
\end{proof}

\begin{Cor}
$$
\sum_{k=1}^n {n \choose k} \frac{(-1)^{k+1}}{k+1} =\frac{n}{n+1}.
$$
$$
H_n:=\sum_{k=1}^n \frac{1}{k} = \sum_{k=1}^n {n \choose k}
\frac{(-1)^{k+1}}{k}.
$$
\end{Cor}

\noindent {\bf Proof.} The polynomial identity
$$
\sum_{k=1}^n \frac{(1-x)^k}{k} = \sum_{k=1}^n {n \choose k}\frac{(-1)^k (x^k-1)}{k}.
$$
can be integrated to yield
$$
\sum_{k=1}^n \frac{(1-x)^{k+1}}{k(k+1)} = \sum_{k=1}^n {n \choose k} \frac{(-1)^{k+1}}{k} \left(\frac{x^{k+1}}{k+1} - x\right) + C,
$$
where we get $C = \sum_{k=1}^n {n \choose k} \frac{(-1)^{k+1}}{k+1}$
by putting $x=0$. The value at $x=0$ implies the first identity:
$$
\sum_{k=1}^n {n \choose k} \frac{(-1)^{k+1}}{k+1} =\frac{n}{n+1}.
$$
Equating the coefficients of $x$ on both sides of the polynomial
identity gives us the second identity:
$$
H_n:=\sum_{k=1}^n \frac{1}{k} = \sum_{k=1}^n {n \choose k}
\frac{(-1)^{k+1}}{k}.
$$

\section{Sums of higher harmonic numbers}

\noindent We obtained an identity for harmonic numbers $H_n$ above.
In this section, we prove more general identities for the harmonic
numbers of higher order. Let
$$H_k^{(r)}:=\sum_{i=1}^k 1/i^r$$ be the
$k$-th harmonic number of order $r$. In \cite{Mneimneh}, for a
positive integer $n$ and $0\le q\le 1$, it is shown that
\begin{equation}
\sum_{k=0}^n H_k\binom{n}{k}(1-q)^k q^{n-k}=H_n-\sum_{j=1}^n\frac{q^j}{j}\,.
\label{eq:mneimneh}
\end{equation}
This relation is derived by the author from an interesting
probabilistic analysis. In this section, we obtain formulae
generalizing (\ref{eq:mneimneh}).

\begin{theorem}
\begin{align}
&\sum_{k=0}^n H_k^{(r)}\binom{n}{k}(1-q)^k q^{n-k}\notag\\
&=H_n^{(r)}-\sum_{j=1}^n\left(\sum_{l=0}^{j-1}(-1)^{j-l-1}\binom{n-l-1}{n-j}\binom{n}{l}\frac{1}{(n-l)^{r-1}}\right)\frac{q^j}{j}\,.
\label{eq:101}
\end{align}
\label{th:rrr}
\end{theorem}

In particular, when $r=1$, we find the following relation. Thus, the formula (\ref{eq:mneimneh}) is recovered.

\begin{Lem}
\begin{equation}
\sum_{l=0}^{j-1}(-1)^{j-l-1}\binom{n-l-1}{n-j}\binom{n}{l}=1\,.
\label{eq:1r1}
\end{equation}
\label{lem:r1}
\end{Lem}

When $r=2$, we find the following relation. Here, $(n)_j=n(n-1)\cdots(n-j+1)$ ($j\ge 1$) is the falling factorial with $(n)_0=1$, and $\stf{n}{k}$ denotes the (unsigned) Stirling number of the first kind, arising from the relation
$(x)_n=\sum_{k=0}^n(-1)^{n-k}\stf{n}{k}x^k$.

\begin{Lem}
\begin{multline}
\sum_{l=0}^{j-1}(-1)^{j-l-1}\binom{n-l-1}{n-j}\binom{n}{l}\frac{1}{n-l}=H_n-H_{n-j}\\
=\frac{1}{(n)_j}\sum_{\nu=0}^{j-1}(-1)^{j-\nu-1}(\nu+1)\stf{j}{\nu+1}n^\nu\,.
\label{eq:1r2}
\end{multline}
\label{lem:r2}
\end{Lem}

Note that
$$
\binom{n-l-1}{n-j}\binom{n}{l}\frac{1}{n-l}\ne \frac{l+1}{(n)_j}\stf{j}{l+1}n^l\,.
$$
Hence, we have the following formula.

\begin{Cor}
\begin{multline}
\sum_{k=0}^n H_k^{(2)}\binom{n}{k}(1-q)^k q^{n-k}=H_n^{(2)}-\sum_{j=1}^n(H_n-H_{n-j})\frac{q^j}{j}\\
=H_n^{(2)}-\sum_{j=1}^n\left(\frac{1}{(n)_j}\sum_{\nu=0}^{j-1}(-1)^{j-\nu-1}(\nu+1)\stf{j}{\nu+1}n^\nu\right)\frac{q^j}{j}\,.
\label{eq:104}
\end{multline}
\label{cor:r2}
\end{Cor}

\begin{proof}[Proof of Theorem \ref{th:rrr}.]
We shall show
\begin{align}
&\sum_{k=0}^n H_k^{(r)}\binom{n}{k}(1-q)^k q^{n-k}\notag\\
&=H_n^{(r)}-\sum_{j=1}^n\binom{n}{j}\left(\sum_{l=0}^{j-1}(-1)^{j-l-1}\binom{j-1}{l}\frac{1}{(n-l)^r}\right)q^j\,.
\label{eq:100}
\end{align}
We have
\begin{align*}
&\sum_{k=0}^n H_k^{(r)}\binom{n}{k}(1-q)^k q^{n-k}
=\sum_{k=0}^n H_k^{(r)}\binom{n}{k}\sum_{l=0}^k(-1)^{k-l}\binom{k}{l}q^{n-l}\\
&=\sum_{l=0}^n q^{n-l}\binom{n}{l}\sum_{k=l}^n(-1)^{k-l}\binom{n-l}{n-k}H_k^{(r)}\\
&=\sum_{j=0}^n\binom{n}{j}q^j\sum_{\nu=0}^j(-1)^{j-\nu}\binom{j}{\nu}H_{n-\nu}^{(r)}\\
&=H_n^{(r)}-\sum_{j=1}^n(-1)^{j-1}\binom{n}{j}q^j\sum_{\nu=0}^j(-1)^{\nu}\binom{j}{\nu}\sum_{l=0}^{n-1}\frac{1}{(n-l)^r}\,.
\end{align*}
Since
$$
\sum_{\nu=0}^l(-1)^\nu\binom{j}{\nu}=(-1)^l\binom{j-1}{l}\quad(\text{proved by induction on $l(\ge 0)$})
$$
and
$$
\sum_{\nu=0}^j(-1)^\nu\binom{j}{\nu}=(1-1)^j=0\,,
$$
we have
\begin{align*}
&\sum_{\nu=0}^j(-1)^{\nu}\binom{j}{\nu}\sum_{l=0}^{n-1}\frac{1}{(n-l)^r}\\
&=\sum_{l=0}^{j-1}\left(\sum_{\nu=0}^l(-1)^\nu\binom{j}{\nu}\right)\frac{1}{(n-l)^r}+\sum_{l=j}^{n-1}\left(\sum_{\nu=0}^j(-1)^\nu\binom{j}{\nu}\right)\frac{1}{(n-l)^r}\\
&=\sum_{l=0}^{j-1}(-1)^l\binom{j-1}{l}\frac{1}{(n-l)^r}\,.
\end{align*}
By (\ref{eq:100}), it is straightforward to get (\ref{eq:101}).
\end{proof}

\begin{proof}[Proof of Lemma \ref{lem:r1}.]
Put
\begin{align*}
A(n,j)&=\sum_{l=0}^{j-1}(-1)^{j-l-1}\binom{n-l-1}{n-j}\binom{n}{l}\,,\\
B(n,j)&=\sum_{l=0}^{j}(-1)^{j-l}\binom{n-l}{n-j}\binom{n}{l}\,.
\end{align*}
Since,
\begin{align*}
B(n+1,j)&=\frac{n+1}{n-j+1}B(n,j)=\frac{(n+1)n}{(n-j+1)(n-j)}B(n-1,j)\\
&=\cdots=\frac{(n+1)n\cdots(j+1)}{(n-j+1)!}B(j,j)\\
&=\frac{(n+1)!}{(n-j+1)!j!}\sum_{l=0}^j(-1)^{j-l}\binom{j}{l}\\
&=\binom{n+1}{j}(1-1)^j=0\,,
\end{align*}
we have
$$
A(n,j+1)-A(n,j)=B(n,j)=0\,.
$$
Hence, we obtain
$$
A(n,j)=A(n,j-1)=\cdots=A(n,1)=(-1)^0\binom{n-1}{n-1}\binom{n}{0}=1\,.
$$
\end{proof}

\begin{proof}[Proof of Lemma \ref{lem:r2}.]
The formula (\ref{eq:1r2}) is yielded from the definition of the Stirling numbers of the first kind:
\begin{align*}
(x)_j&=\sum_{k=0}^j(-1)^{j-k}\stf{j}{k}x^k\\
&=\sum_{\nu=0}^{j-1}(-1)^{j-\nu-1}\stf{j}{\nu+1}x^{\nu+1}\quad(\text{if}\, j\ge 1)\,.
\end{align*}
Differentiating both sides with respect to $x$ gives
$$
(x)_j\sum_{l=0}^{j-1}\frac{1}{x-l}=\sum_{\nu=0}^{j-1}(-1)^{j-\nu-1}(\nu+1)\stf{j}{\nu+1}x^\nu\,.
$$
Thus, the right-hand side of (\ref{eq:1r2}) is equal to
$$
\sum_{l=0}^{j-1}\frac{1}{n-l}=H_n-H_{n-j}\,.
$$
Put the left-hand side of (\ref{eq:1r2}) as
$$
C(n,j):=\sum_{l=0}^{j-1}(-1)^{j-l-1}\binom{n-l-1}{n-j}\binom{n}{l}\frac{1}{n-l}\,.
$$
Then
\begin{align*}
&C(n,j)-C(n,j-1)\\
&=\sum_{l=0}^{j-2}(-1)^{j-l-1}\left(\binom{n-l-1}{n-j}\binom{n}{l}+\binom{n-l-1}{n-j+1}\binom{n}{l}\right)\frac{1}{n-l}\\
&\quad +\binom{n}{j-1}\frac{1}{n-j+1}\\
&=\sum_{l=0}^{j-1}(-1)^{j-l-1}\binom{n-l}{n-j+1}\binom{n}{l}\frac{1}{n-l}\\
&=\binom{n}{j-1}\sum_{l=0}^{j-1}(-1)^{j-l-1}\binom{j-1}{l}\frac{1}{n-l}\,.
\end{align*}
Now,
\begin{align*}
&\sum_{l=0}^{j-1}(-1)^{j-l-1}\binom{j-1}{l}\frac{1}{n-l}\\
&=\left.\int\sum_{l=0}^{j-1}(-1)^{j-l-1}\binom{j-1}{l}x^{n-l-1}d x\right|_{x=1}\\
&=\left.\int x^{n-1}\left(1-\frac{1}{x}\right)^{j-1}d x\right|_{x=1}\\
&=\left.\left(1-\frac{1}{x}\right)^{j}\frac{{}_2 F_1(-j+1,n-j+1;n-j+2;x)}{(x-1)^j(n-j+1)}\right|_{x=1}\\
&=\frac{\Gamma(j)\Gamma(n-j+2)}{(n-j+1)\Gamma(n+1)}=\frac{(j-1)!(n-j)!}{n!}\,,
\end{align*}
where ${}_2 F_1(a,b;c;z)$ is the Gauss hypergeometric function.
Hence,
$$
C(n,j)-C(n,j-1)=\frac{1}{n-j+1}\,.
$$
Therefore,
\begin{align}
C(n,j)&=C(n,j-1)+\frac{1}{n-j+1}\notag\\
&=C(n,j-2)+\frac{1}{n-j+2}+\frac{1}{n-j+1}\notag\\
&=\cdots\notag\\
&=C(n,1)+\frac{1}{n-1}+\cdots+\frac{1}{n-j+2}+\frac{1}{n-j+1}\notag\\
&=\sum_{l=0}^{j-1}\frac{1}{n-l}\,.
\label{eq:r333}
\end{align}
\end{proof}

\subsection{The case $r=3$}

When $r=3$, we have the following.

\begin{Prop}
\begin{align*}
&\sum_{l=0}^{j-1}(-1)^{j-l-1}\binom{n-l-1}{n-j}\binom{n}{l}\frac{1}{(n-l)^2}\\
&=\frac{(H_n-H_{n-j})^2}{2}+\frac{H_n^{(2)}-H_{n-j}^{(2)}}{2}\,.
\end{align*}
\label{prp:r3}
\end{Prop}

Therefore, we have the following formula.

\begin{Cor}
\begin{align*}
&\sum_{k=0}^n H_k^{(3)}\binom{n}{k}(1-q)^k q^{n-k}\\
&=H_n^{(3)}-\sum_{j=1}^n\left(\frac{(H_n-H_{n-j})^2}{2}+\frac{H_n^{(2)}-H_{n-j}^{(2)}}{2}\right)\frac{q^j}{j}\,.
\end{align*}
\label{cor:r2}
\end{Cor}

\begin{proof}[Proof of Proposition \ref{prp:r3}.]
Put
$$
D(n,j):=\sum_{l=0}^{j-1}(-1)^{j-l-1}\binom{n-l-1}{n-j}\binom{n}{l}\frac{1}{(n-l)^2}\,.
$$
Then
\begin{equation}
D(n,j)-D(n,j-1)=\binom{n}{j-1}\sum_{l=0}^{j-1}(-1)^{j-l-1}\binom{j-1}{l}\frac{1}{(n-l)^2}\,.
\label{eq:rr44}
\end{equation}
We shall prove that
\begin{equation}
D(n,j)-D(n,j-1)
=\frac{1}{n-j+1}(H_n-H_{n-j})\,.
\label{eq:rr4}
\end{equation}
By (\ref{eq:rr4}), we get
\begin{align*}
D(n,j)&=\left(\frac{H_n}{n-j+1}-\frac{H_{n-j}}{n-j+1}\right)+\left(\frac{H_n}{n-j+2}-\frac{H_{n-j+1}}{n-j+2}\right)\\
&\quad +\cdots+\left(\frac{H_n}{n-1}-\frac{H_{n-2}}{n-1}\right)+\left(\frac{H_n}{n}-\frac{H_{n-1}}{n}\right)\\
&=H_n\left(\frac{1}{n-j+1}+\frac{1}{n-j+2}+\cdots+\frac{1}{n-1}+\frac{1}{n}\right)\\
&\quad-\left(\frac{H_{n-j}}{n-j+1}+\frac{H_{n-j+1}}{n-j+2}+\cdots+\frac{H_{n-2}}{n-1}+\frac{H_{n-1}}{n}\right)\\
&=H_n(H_n-H_{n-j})-\left(\frac{H_n^2-H_n^{(2)}}{2}-\frac{H_{n-j}^2-H_{n-j}^{(2)}}{2}\right)\\
&=\frac{(H_n-H_{n-j})^2}{2}+\frac{H_n^{(2)}-H_{n-j}^{(2)}}{2}\,.
\end{align*}
In order to prove (\ref{eq:rr4}), we put
$$
E(n,j)=(n-j+1)\bigl(D(n,j)-D(n,j-1)\bigr)\,.
$$
Then by (\ref{eq:rr44}) and Lemma \ref{lem:r1} (\ref{eq:1r1}), we have
\begin{align*}
&E(n,j)-E(n,j-1)\\
&=\frac{1}{n-j+1}\sum_{l=0}^{j-2}(-1)^{j-l-1}\binom{n-l-1}{n-j}\binom{n}{l}+\frac{1}{n-j+1}\binom{n}{j-1}\\
&=\frac{1}{n-j+1}\sum_{l=0}^{j-1}(-1)^{j-l-1}\binom{n-l-1}{n-j}\binom{n}{l}=\frac{1}{n-j+1}\,.
\end{align*}
Hence, by $D(n,1)=1/n^2$, we get
\begin{align*}
&D(n,j)-D(n,j-1)=
\frac{E(n,j)}{n-j+1}\\
&=\frac{1}{n-j+1}\left(\frac{1}{n-j+1}+\frac{1}{n-j+2}+\cdots+\frac{1}{n-1}+E(n,1)\right)\\
&=\frac{1}{n-j+1}(H_n-H_{n-j})\,.
\end{align*}
which is the right-hand side of (\ref{eq:rr4}).
\end{proof}

\subsection{The case $r=4$}

When $r=4$, we have the following.

\begin{Prop}
\begin{align*}
&\sum_{l=0}^{j-1}(-1)^{j-l-1}\binom{n-l-1}{n-j}\binom{n}{l}\frac{1}{(n-l)^3}\\
&=\frac{(H_n-H_{n-j})^3}{6}+\frac{(H_n-H_{n-j})(H_n^{(2)}-H_{n-j}^{(2)})}{2}+\frac{H_n^{(3)}-H_{n-j}^{(3)}}{3}\,.
\end{align*}
\label{prp:r4}
\end{Prop}

Similarly to the case $r=3$, put
$$
\mathfrak D(n,j):=\sum_{l=0}^{j-1}(-1)^{j-l-1}\binom{n-l-1}{n-j}\binom{n}{l}\frac{1}{(n-l)^3}\,.
$$
Then
$$
\mathfrak D(n,j)-\mathfrak D(n,j-1)=\binom{n}{j-1}\sum_{l=0}^{j-1}(-1)^{j-l-1}\binom{j-1}{l}\frac{1}{(n-l)^3}\,.
$$
Put
$$
\mathfrak E(n,j)=(n-j+1)\bigl(\mathfrak D(n,j)-\mathfrak D(n,j-1)\bigr)\,.
$$
Then by (\ref{eq:r333}), we have
\begin{align*}
&\mathfrak E(n,j)-\mathfrak E(n,j-1)\\
&=\frac{1}{n-j+1}\sum_{l=0}^{j-1}(-1)^{j-l-1}\binom{n-l-1}{n-j}\binom{n}{l}\frac{1}{n-l}\\
&=\frac{C(n,j)}{n-j+1}=\frac{1}{n-j+1}\sum_{\ell=0}^{j-1}\frac{1}{n-\ell}=\frac{H_n-H_{n-j}}{n-j+1}\,.
\end{align*}
Thus, by $\mathfrak E(n,1)=1/n^2$, we get
\begin{align*}
&\mathfrak E(n,j)\\
&=\frac{H_n-H_{n-j}}{n-j+1}+\frac{H_n-H_{n-j+1}}{n-j+2}+\cdots+\frac{H_n-H_{n-2}}{n-1}+E(n,1)\\
&=H_n(H_n-H_{n-j})-\left(\frac{H_n^2-H_n^{(2)}}{2}-\frac{H_{n-j}^2-H_{n-j}^{(2)}}{2}\right)\,.
\end{align*}
Hence, by $\mathfrak D(n,1)=1/n^3$, we have
\begin{align*}
&\mathfrak D(n,j)-\mathfrak D(n,j-1)=
\frac{\mathfrak E(n,j)}{n-j+1}\\
&=\frac{H_n(H_n-H_{n-j})}{n-j+1}-\frac{1}{n-j+1}\left(\frac{H_n^2-H_n^{(2)}}{2}-\frac{H_{n-j}^2-H_{n-j}^{(2)}}{2}\right)\,.
\end{align*}
Therefore,
\begin{align*}
&\mathfrak D(n,j)\\
&=\frac{H_n^2+H_n^{(2)}}{2(n-j+1)}-\frac{H_n H_{n-j}}{n-j+1}+\frac{H_{n-j}^2-H_{n-j}^{(2)}}{2(n-j+1)}\\
&\quad +\frac{H_n^2+H_n^{(2)}}{2(n-j+2)}-\frac{H_n H_{n-j+1}}{n-j+2}+\frac{H_{n-j+1}^2-H_{n-j+1}^{(2)}}{2(n-j+2)}\\
&\quad +\cdots\\
&\quad +\frac{H_n^2+H_n^{(2)}}{2(n-1)}-\frac{H_n H_{n-2}}{n-1}+\frac{H_{n-2}^2-H_{n-2}^{(2)}}{2(n-1)}\\
&\quad +\mathfrak D(n,1)\\
&=\frac{(H_n^2+H_n^{(2)})(H_n-H_{n-j})}{2}-H_n\left(\frac{H_n^2-H_n^{(2)}}{2}-\frac{H_{n-j}^2-H_{n-j}^{(2)}}{2}\right)\\
&\quad +\sum_{1\le i_1<i_2<i_3\le n}\frac{1}{i_1 i_2 i_3}-\sum_{1\le i_1<i_2<i_3\le n-j}\frac{1}{i_1 i_2 i_3}\,.
\end{align*}
Since for $n\ge 3$
\begin{align*}
H_n^3&=H_n^{(3)}+6\sum_{1\le i_1<i_2<i_3\le n}\frac{1}{i_1 i_2 i_3}\\
&\quad +3\left(\frac{1}{1^2}\left(H_n-\frac{1}{1}\right)+\frac{1}{2^2}\left(H_n-\frac{1}{2}\right)+\cdots+\frac{1}{n^2}\left(H_n-\frac{1}{n}\right)\right)\\
&=H_n^{(3)}+6\sum_{1\le i_1<i_2<i_3\le n}\frac{1}{i_1 i_2 i_3}+3(H_n^{(2)}H_n-H_n^{(3)})\,.
\end{align*}
we have
$$
\sum_{1\le i_1<i_2<i_3\le n}\frac{1}{i_1 i_2 i_3}=\frac{1}{6}(H_n^3+2 H_n^{(3)}-3 H_n^{(2)}H_n)\,.
$$
Therefore,
\begin{align*}
&\mathfrak D(n,j)\\
&=\frac{(H_n^2+H_n^{(2)})(H_n-H_{n-j})}{2}-H_n\left(\frac{H_n^2-H_n^{(2)}}{2}-\frac{H_{n-j}^2-H_{n-j}^{(2)}}{2}\right)\\
&\quad +\frac{1}{6}(H_n^3+2 H_n^{(3)}-3 H_n^{(2)}H_n-H_{n-j}^3-2 H_{n-j}^{(3)}+3 H_{n-j}^{(2)}H_{n-j})\\
&=\frac{(H_n-H_{n-j})^3}{6}+\frac{(H_n-H_{n-j})(H_n^{(2)}-H_{n-j}^{(2)})}{2}+\frac{H_n^{(3)}-H_{n-j}^{(3)}}{3}\,.
\end{align*}

\subsection{Speculations on harmonic numbers and partitions}

In general, let us put
$$
\mathcal
D_r(n,j):=\sum_{l=0}^{j-1}\binom{n-l-1}{n-j}\binom{n}{l}\frac{1}{(n-l)^{r-1}}.
$$
Then, arguing similarly to the above cases where $1\le r\le 4$, we
have
\begin{align*}
&\mathcal D_5(n,j)=\frac{(H_n-H_{n-j})^4}{4!}+\frac{(H_n-H_{n-j})^2(H_n^{(2)}-H_{n-j}^{(2)})}{4}\\
&\quad +\frac{(H_n-H_{n-j})(H_n^{(3)}-H_{n-j}^{(3)})}{3}+\frac{(H_n^{(2)}-H_{n-j}^{(2)})^2}{4}+\frac{H_n^{(4)}-H_{n-j}^{(4)}}{8}\,,\\
&\mathcal D_6(n,j)=\frac{(H_n-H_{n-j})^5}{5!}+\frac{(H_n-H_{n-j})^3(H_n^{(2)}-H_{n-j}^{(2)})}{12}\\
&\quad +\frac{(H_n-H_{n-j})^2(H_n^{(3)}-H_{n-j}^{(3)})}{6}+\frac{(H_n-H_{n-j})(H_n^{(2)}-H_{n-j}^{(2)})^2}{8}\\
&\quad +\frac{(H_n-H_{n-j})(H_n^{(4)}-H_{n-j}^{(4)}}{4}+\frac{(H_n^{(2)}-H_{n-j}^{(2)})(H_n^{(3)}-H_{n-j}^{(3)})}{6}\\
&\quad +\frac{H_n^{(5)}-H_{n-j}^{(5)}}{5}\,,\\
&\mathcal D_7(n,j)=\frac{(H_n-H_{n-j})^6}{6!}+\frac{(H_n-H_{n-j})^4(H_n^{(2)}-H_{n-j}^{(2)})}{48}\\
&\quad +\frac{(H_n-H_{n-j})^3(H_n^{(3)}-H_{n-j}^{(3)})}{18}+\frac{(H_n-H_{n-j})^2(H_n^{(2)}-H_{n-j}^{(2)})^2}{16}\\
&\quad +\frac{(H_n-H_{n-j})^2(H_n^{(4)}-H_{n-j}^{(4)})}{8}+\frac{(H_n-H_{n-j})(H_n^{(2)}-H_{n-j}^{(2)})(H_n^{(3)}-H_{n-j}^{(3)})}{6}\\
&\quad +\frac{(H_n-H_{n-j})(H_n^{(5)}-H_{n-j}^{(5)})}{5}+\frac{(H_n^{(2)}-H_{n-j}^{(2)})^3}{48}\\
&\quad +\frac{(H_n^{(2)}-H_{n-j}^{(2)})(H_n^{(4)}-H_{n-j}^{(4)})}{8}+\frac{(H_n^{(3)}-H_{n-j}^{(3)})^2}{18}+\frac{H_n^{(6)}-H_{n-j}^{(6)}}{6}\,.
\end{align*}
It is interesting to observe that the number of terms of each of the right-hand sides of $\mathcal D_r(n,j)$ is equal to the number of partitions of $r$ ($r=1,2,3,4,5,6,7$), respectively. In addition, the same terms of generalized harmonic numbers appear in \cite{Choi,Hoffman}:
\begin{align*}
&\sum_{n=1}^\infty\frac{H_n}{(n+1)(n+2)}=1\,,\\
&\sum_{n=1}^\infty\frac{(H_n)^2-H_n^{(2)}}{2(n+1)(n+2)}=1\,,\\
&\sum_{n=1}^\infty\frac{(H_n)^3-3 H_n H_n^{(2)}+2 H_n^{(3)}}{3!(n+1)(n+2)}=1\,,\\
&\sum_{n=1}^\infty\frac{(H_n)^4-6(H_n)^2 H_n^{(2)}+8 H_n H_n^{(3)}+3(H_n^{(2)})^2-6 H_n^{(4)}}{4!(n+1)(n+2)}\,,\\
&\sum_{n=1}^\infty\frac{1}{5!(n+1)(n+2)}\biggl((H_n)^5-10(H_n)^3 H_n^{(2)}+20(H_n)^2 H_n^{(3)})^2\\
&\quad +15 H_n(H_n^{(2)}-30 H_n H_n^{(4)}-20 H_n^{(2)}H_n^{(3)}+24 H_n^{(5)}\biggr)=1\,,\\
&\sum_{n=1}^\infty\frac{1}{6!(n+1)(n+2)}\biggl((H_n)^6-15(H_n)^4 H_n^{(2)}+40(H_n)^3 H_n^{(3)}\\
&\quad +45(H_n)^2(H_n^{(2)})^2-90(H_n)^2 H_n^{(4)}-120 H_n H_n^{(2)}H_n^{(3)}+144 H_n H_n^{(5)}\\
&\quad -15(H_n^{(2)})^3+90 H_n^{(2)}H_n^{(4)}+40(H_n^{(3)})^2-120 H_n^{(5)}
\biggr)=1\,.
\end{align*}

No simple closed form has been found, but $\mathcal D_r(n,j)$ ($r\ge 2$) can be expressed by a combinatorial sum (\cite[Proposition 1 (17)]{Hoffman}):
\begin{multline*}
\mathcal D_{r+1}(n,j)
=\sum_{i_1+2 i_2+3 i_3+\cdots=r}\frac{1}{i_1!i_2!i_3!\cdots}\\
\times\left(\frac{H_n-H_{n-j}}{1}\right)^{i_1}\left(\frac{H_n^{(2)}-H_{n-j}^{(2)}}{2}\right)^{i_2}\left(\frac{H_n^{(3)}-H_{n-j}^{(3)}}{3}\right)^{i_3}\cdots
\end{multline*}
and in terms of the determinant (\cite[Ch. I \S 2]{MacDonald}):
\begin{multline*}
\mathcal D_{r+1}(n,j)=\frac{1}{r!}\\
\times\left|
\begin{array}{ccccc}
H_n-H_{n-j}&-1&0&\cdots&0\\
H_n^{(2)}-H_{n-j}^{(2)}&H_n-H_{n-j}&-2&\cdots&0\\
\vdots&\vdots&\vdots&\ddots&\vdots\\
H_n^{(r-1)}-H_{n-j}^{(r-1)}&H_n^{(r-2)}-H_{n-j}^{(r-2)}&H_n^{(r-3)}-H_{n-j}^{(r-3)}&\cdots&-r+1\\
H_n^{(r)}-H_{n-j}^{(r)}&H_n^{(r-1)}-H_{n-j}^{(r-1)}&H_n^{(r-2)}-H_{n-j}^{(r-2)}&\cdots&H_n-H_{n-j}
\end{array}
\right|\,.
\end{multline*}
See also \cite{ChCh,CoCa}.

\section*{Acknowledgments}
This work was done during the first author's visit to the Indian Statistical Institute Bangalore, India in July-August 2023. He is grateful for his second author's hospitality.

\end{document}